\documentclass[11pt,reqno]{amsart}

\usepackage{amssymb}
\usepackage{amsmath,amsthm,amsfonts,amssymb,latexsym,mathrsfs,color,hyperref,fancybox}
\usepackage{graphicx}
\usepackage{xspace}
\usepackage{multirow}
\usepackage{diagbox}
\usepackage{capt-of}
\usepackage{tikz}
\tikzset{
	level 1/.style = {sibling distance = 1.5cm},
	level 2/.style = {sibling distance = 0.8cm},
    level distance = 1.0 cm
}
\usetikzlibrary {decorations.pathmorphing}
\tikzstyle{snakeline} = [decorate, decoration={snake, amplitude=.4mm, segment length=2mm}]

\newtheorem{theorem}{Theorem}
\newtheorem{corollary}[theorem]{Corollary}
\newtheorem{proposition}[theorem]{Proposition}

\newtheorem{lemma}[theorem]{Lemma}
\newtheorem{definition}[theorem]{Definition}

\newtheorem{problem}[theorem]{Problem}

\newtheorem*{HBtheorem}{Hermite-Biehler Theorem}

\newcommand{\mdd}{\mathcal{D}}
\newcommand{\even}{{\rm even\,}}
\newcommand{\negg}{{\rm neg\,}}
\newcommand{\Plat}{{\rm Plat}}
\newcommand{\Dplat}{{\rm Dplat}}
\newcommand{\Pasc}{{\rm Pasc}}
\newcommand{\Rpd}{{\rm Rpd}}
\newcommand{\Eudd}{{\rm Eud}}
\newcommand{\Dasc}{{\rm Dasc}}
\newcommand{\Lap}{{\rm Lap}}
\newcommand{\dplat}{{\rm dplat}}
\newcommand{\pasc}{{\rm pasc}}

\newcommand{\lends}{{\rm exl}}
\newcommand{\mends}{{\rm exm}}
\newcommand{\rends}{{\rm exr}}
\newcommand{\eudd}{{\rm eud}}

\newcommand{\rpd}{{\rm rpd}}

\newcommand{\Ddes}{{\rm Ddes\,}}

\newcommand{\Des}{{\rm Des\,}}

\newcommand{\Asc}{{\rm Asc\,}}

\newcommand{\udrun}{{\rm udrun\,}}

\newcommand{\dasc}{{\rm dasc\,}}

\newcommand{\plat}{{\rm plat\,}}
\newcommand{\ap}{{\rm ap\,}}
\newcommand{\lap}{{\rm lap\,}}

\newcommand{\ddes}{{\rm ddes\,}}
\newcommand{\des}{{\rm des\,}}

\newcommand{\exc}{{\rm exc\,}}

\newcommand{\mtn}{\mathcal{T}}

\newcommand{\msn}{\mathfrak{S}_n}

\newcommand{\ms}{\mathfrak{S}}

\newcommand{\lrf}[1]{\lfloor #1\rfloor}

\newcommand{\mq}{\mathcal{Q}}
\newcommand{\mqn}{\mathcal{Q}_n}

\newcommand{\asc}{{\rm asc\,}}

\DeclareMathOperator{\R}{\mathbb{R}}

\newcommand{\rz}{{\rm RZ}}

\newcommand{\Stirling}[2]{\genfrac{\{}{\}}{0pt}{}{#1}{#2}}

\title{Stirling permutation codes. II}
\author[S.-M.~Ma]{Shi-Mei Ma}
\address{School of Mathematics and Statistics, Shandong University of Technology, Zibo 255049, Shandong, China}
\email{shimeimapapers@163.com (S.-M. Ma)}
\author[H.~Qi]{Hao Qi}
\address{College of Mathematics and Physics, Wenzhou University, Wenzhou 325035, P.R. China}
\email{qihao@wzu.edu.cn (H.~Qi)}
\author{Jean Yeh}
\address{Department of Mathematics, National Kaohsiung Normal University, Kaohsiung 82444, Taiwan}
\email{chunchenyeh@nknu.edu.tw (J. Yeh)}
\author[Y.-N. Yeh]{Yeong-Nan Yeh}
\address{College of Mathematics and Physics, Wenzhou University, Wenzhou 325035, P.R. China}
\email{mayeh@math.sinica.edu.tw (Y.-N. Yeh)}
\begin{document}

\maketitle
\begin{abstract}
In the context of Stirling polynomials, Gessel and Stanley introduced the definition of
Stirling permutation, which has attracted extensive attention over the past decades.
Recently, we introduced Stirling permutation code and provided
numerous equidistribution results as applications. The purpose of the present work is to
further analyse Stirling permutation code.
First, we derive an expansion formula expressing the joint distribution of the types $A$ and $B$ descent statistics over the hyperoctahedral group, and
we also find an interlacing property involving the zeros of its coefficient polynomials. Next, we
prove a strong connection between signed permutations in the hyperoctahedral group and Stirling permutations.
Furthermore, we investigate unified generalizations
of the trivariate second-order Eulerian polynomials and ascent-plateau polynomials.
Using Stirling permutation codes, we provide expansion formulas for eight-variable and
seventeen-variable polynomials, which imply several $e$-positive expansions and clarify the connections among several statistics.
Our results generalize the results of B\'ona, Chen-Fu, Dumont, Janson, Haglund-Visontai and Petersen.
\bigskip

\end{abstract}
\date{\today}
\tableofcontents
\section{Introduction}
Let $[n]=\{1,2,\ldots,n\}$ and let $\pm[n]=[n]\cup\{\overline{1},\overline{2},\ldots,\overline{n}\}$, where $\overline{i}=-i$.
The {\it symmetric group} $\msn$ is the group of all permutations on $[n]$, and the {\it hyperoctahedral group} $\msn^B$ is the group of signed permutations on $\pm[n]$ with the property that $\pi\left(\overline{i}\right)=-\pi(i)$ for all $i\in [n]$.
Let $\pi=\pi(1)\pi(2)\cdots \pi(n)\in\msn^B$, and let $\negg(\pi)$ be the number of {\it negative entries} of $\pi$.
The numbers of {\it types $A$ and $B$ descent} statistics of $\pi$ are respectively defined by
$$\operatorname{des}_A(\pi)=\#\{i\in\{1,\ldots,n-1\}\mid \pi(i)>\pi(i+1)\},$$
$$\operatorname{des}_B(\pi)=\#\{i\in\{0,1,\ldots,n-1\}\mid \pi(i)>\pi(i+1),~\pi(0)=0\}.$$

In~\cite{Brenti94}, Brenti studied the following {\it Eulerian polynomial of type $B$} and its $q$-analog:
\begin{equation*}\label{BnxDef}
B_n(x)=\sum_{\pi\in \msn^B}x^{\operatorname{des}_B(\pi)},~B_n(x,q)=\sum_{\pi\in \msn^B}x^{\operatorname{des}_B(\pi)}q^{\negg(\pi)}.
\end{equation*}
In particular, $B_n(x,0)=A_n(x)=\sum_{\pi\in\msn}x^{\operatorname{des}_A(\pi)}$, where $A_n(x)$ is usually named as {\it the Eulerian polynomial of type $A$}. Since then, there has been a growing interest in the similar properties of Eulerian-type polynomials, including unimodality, real-rootedness, $\gamma$-positivity as well as algebraic and geometric interpretations,
see~\cite{Gessel20,Han2021,Lin21,Ma19,Petersen15,Zhuang16,Zhuang17} for instances.

A remarkable result of Foata-Sch\"utzenberger~\cite{Foata70} says that
\begin{equation*}
A_n(x)=\sum_{k=0}^{\lrf{({n-1})/{2}}}\gamma_1(n,k)x^k(1+x)^{n-1-2k}.
\end{equation*}
Let $\gamma_2(n,k)$ (resp.~$\gamma_3(n,k)$) be the number of permutations in $\msn$ with $k$ interior peaks (resp.~left peaks).
By introducing modified Foata-Strehl action, Br\"and\'{e}n~\cite{Branden08} deduced that
\begin{equation}\label{Anx-gamma}
A_n(x)=\sum_{k=0}^{\lrf{({n-1})/{2}}}\frac{1}{2^{n-1-2k}}\gamma_2(n,k)x^k(1+x)^{n-1-2k}.
\end{equation}
Using the theory of enriched $P$-partitions, Petersen~\cite[Proposition~4.15]{Petersen07} obtained the following $\gamma$-positive expansion:
\begin{equation}\label{Bnxgamma}
B_n(x)=\sum_{i=0}^{\lrf{n/2}}4^i\gamma_3(n,i)x^i(1+x)^{n-2i}.
\end{equation}

Consider the polynomials
$$b_n(x,y)=\sum_{\pi\in\msn^B}x^{\operatorname{des}_A(\pi)}y^{\operatorname{des}_B(\pi)}.$$
Clearly, $b_n(x,1)=2^nA_n(x)$ and $b_n(1,x)=B_n(x)$. Following~\cite{Adin01}, the number of {\it flag descents} of $\pi\in\msn^B$ can be defined by $\operatorname{des}_A(\pi)+\operatorname{des}_B(\pi)$. Hence $b_n(x,x)$ is the flag descent polynomial.
The {\it up-down runs} of $\pi\in\msn$ are the maximal consecutive subsequence that is increasing or decreasing of $\pi$ endowed with a 0
in the front~(see~\cite{Ma22,Stanley08,Zhuang16} for details). Let $\udrun(\pi)$ denote the number of up-down runs of $\pi$. For example, $\udrun(623415)=\udrun(0623415)=5$.
We now present the first result of this paper, which gives a unified extension of~\eqref{Anx-gamma} and~\eqref{Bnxgamma}.
\begin{theorem}\label{thm055}
For any $n\geqslant 2$, the bivariate polynomial $b_n(x,y)$ has the expansion formula:
\begin{equation*}\label{System-recu003}
b_n(x,y)=(1+y)\sum_{k\geqslant 0}4^k\xi(n,k)(xy)^k(1+xy)^{n-1-2k}+y(1+x)\sum_{\ell\geqslant 0}4^k\zeta(n,k)(xy)^{k}(1+xy)^{n-2-2k},
\end{equation*}
where $\xi(n,k)=T(n,2k+1),~\zeta(n,k)=2T(n,2k+2)$, and $T(n,k)$ is the number of permutations in the symmetric group $\msn$ with $k$ up-down runs.
\end{theorem}
As illustrations of Theorem~\ref{thm055}, we have $b_1(x,y)=1+y,~b_2(x,y)=(1 + y + x y + x y^2)+(2 y + 2 x y)$,
$$b_3(x,y)=(1 + y + 10 x y + 10 x y^2 + x^2 y^2 + x^2 y^3)+(6 y + 6 x y + 6 x y^2 + 6 x^2 y^2).$$
In Theorem~\ref{thm05}, we shall establish a strong connection
between the joint distribution of $(\operatorname{des}_A,\operatorname{des}_B,\negg)$ over $\msn^B$ and the joint distribution of
$(\lap,\ap,\even)$ over restricted Stirling permutations.

The study of Stirling permutations originated from the work of Ramanujan~\cite{Ramanujan27}, when he considered the Taylor series expansion:
\begin{equation*}\label{eqr}
\mathrm{e}^{nx}=\sum_{r=0}^n\frac{(nx)^r}{r!}+\frac{(nx)^n}{n!}S_n(x).
\end{equation*}
He claimed that
$S_n(1)=\frac{n!}{2}\left(\frac{\mathrm{e}}{n}\right)^n-\frac{2}{3}+\frac{4}{135n}+O(n^{-2})$,
which was independently proved in 1928 by Szeg\"o and Watson.
Buckholtz~\cite{Buckholtz} found that
$$S_n(x)=\sum_{r=0}^{k-1}\frac{1}{n^r}U_r(x)+O(n^{-k}),$$
where $$U_r(x)=(-1)^r\left(\frac{x}{1-x}\frac{\mathrm{d}}{\mathrm{d}x}\right)^r\frac{x}{1-x}=(-1)^r\frac{C_r(x)}{(1-x)^{2r+1}},$$
and $C_r(x)$ is a polynomial of degree $r$. Let $\Stirling{n}{k}$ be the
{\it Stirling number of the second kind}, i.e.,
the number of set partitions of $[n]$ into $k$ blocks.
In~\cite{Carlitz65}, Carlitz discovered that
$$\sum_{k=0}^\infty \Stirling{n+k}{k}x^k=\frac{C_n(x)}{(1-x)^{2n+1}}.$$
The polynomials $C_n(x)$ are now known as the {\it second-order Eulerian polynomials}.

Let $[\mathbf{n}]_2$ denote the multiset $\{1^{2},2^{2},\ldots,n^{2}\}$, where each element $i$ appears $2$ times.
We say that the multipermutation $\sigma$ of $[\mathbf{n}]_2$ is a {\it Stirling permutation}
if $\sigma_i=\sigma_j$, then $\sigma_s>\sigma_i$ for all $i<s<j$.
Let $\mq_n$ denote the set of all Stirling permutations of $[\mathbf{n}]_2$. For example, $\mq_2=\{1122,1221,2211\}$.
Gessel-Stanley~\cite{Gessel78} discovered that $C_n(x)$ are the descent polynomials over all Stirling permutations in $\mq_n$.
Recently, the theory of Stirling permuations
becomes an active research domain, see~\cite{Brualdi20,Janson11,Lin21,Liu21,Ma23}. There are several variants of Stirling permutations,
including Stirling permutations of a general multiset~\cite{Kuba21} and quasi-Stirling permutations~\cite{Elizalde}.

For $\sigma\in\mqn$, except where explicitly stated, we always assume that $\sigma_0=\sigma_{2n+1}=0$.
Let
\begin{align*}
\operatorname{Asc}(\sigma)&=\{\sigma_i\mid \sigma_{i-1}<\sigma_{i}\},~
\operatorname{Plat}(\sigma)=\{\sigma_i\mid \sigma_i=\sigma_{i+1}\},~
\operatorname{Des}(\sigma)=\{\sigma_i\mid \sigma_i>\sigma_{i+1}\},\\
\operatorname{Lap}(\sigma)&=\{\sigma_i\mid \sigma_{i-1}<\sigma_{i}=\sigma_{i+1}\},~
\operatorname{Rpd}(\sigma)=\{\sigma_i\mid \sigma_{i-1}=\sigma_i>\sigma_{i+1}\},\\
\operatorname{Eud}(\sigma)&=\{\sigma_i\mid \sigma_{i-1}<\sigma_i=\sigma_j>\sigma_{j+1},~i<j\},~
\operatorname{Apd}(\sigma)=\{\sigma_i\mid \sigma_{i-1}<\sigma_i=\sigma_{i+1}>\sigma_{i+1}\},\\
\operatorname{Vv}(\sigma)&=\{\sigma_i\mid \sigma_{i-1}>\sigma_i<\sigma_{i+1}, \sigma_{j-1}>\sigma_{j}<\sigma_{j+1},~\sigma_{i}=\sigma_j,~i<j-2\}
\end{align*}
be the sets of ascents, plateaux, descents, left ascent-plateaux, right plateau-descents, exterior up-down-pairs,
ascent-plateau-descents, valley-valley pairs of $\sigma$, respectively.
We use $\asc(\sigma)$, $\plat(\sigma)$, $\des(\sigma)$, $\lap(\sigma)$, $\operatorname{rpd}(\sigma)$, $\operatorname{eud}(\sigma)$, $\operatorname{apd}(\sigma)$ and  $\operatorname{vv}(\sigma)$ to respectively denote the number of ascents, plateaux, descents, left ascent-plateaux, right plateau-descents,
exterior up-down-pairs, ascent-plateau-descents, valley-valley pairs of $\sigma$.
The statistic $\operatorname{vv}$ is a new statistic.
It should be noted that if $\sigma_i$ is a valley-valley pair value, then the two copies of $\sigma_i$ are both valleys.

It is now well known that
$C_n(x)=\sum_{\sigma\in\mqn}x^{\asc{(\sigma)}}=\sum_{\sigma\in\mqn}x^{\plat{(\sigma)}}=\sum_{\sigma\in\mqn}x^{\des{(\sigma)}}$.
The {\it trivariate second-order Eulerian polynomial} is defined by
\begin{align*}
C_n(x,y,z)&=\sum_{\sigma\in\mqn}x^{\asc{(\sigma)}}y^{\plat{(\sigma)}}z^{\des(\sigma)}.
\end{align*}
The study of $C_n(x,y,z)$ was initiated by Dumont~\cite{Dumont80}, who showed that
\begin{equation}\label{Dumont80}
C_{n+1}(x,y,z)=xyz\left(\frac{\partial}{\partial x}+\frac{\partial}{\partial y}+\frac{\partial}{\partial z}\right)C_n(x,y,z),
\end{equation}
which implies that $C_n(x,y,z)$ is symmetric in its variables.
B\'ona~\cite{Bona08} independently found that the plateau statistic $\plat$ is equidistributed with the descent statistic $\des$ over $\mqn$.
The symmetry of the joint distribution $(\asc,\des,\plat)$ was rediscovered by Janson~\cite[Theorem~2.1]{Janson08}. In~\cite{Haglund12}, Haglund-Visontai introduced a refinement of $C_n(x,y,z)$ by indexing each ascent,
descent and plateau according to the values where they appear.
Recently, using the theory of context-free grammars, Chen-Fu~\cite{Chen21}
found the following result.
\begin{proposition}
The trivariate polynomial $C_n(x,y,z)$ is $e$-positive, i.e.,
\begin{equation}\label{Cnxyz}
C_{n}(x,y,z)=\sum_{i+2j+3k=2n+1}\gamma_{n,i,j,k}(x+y+z)^{i}(xy+yz+zx)^{j}(xyz)^k,
\end{equation}
where the coefficient $\gamma_{n,i,j,k}$ equals the number of 0-1-2-3 increasing plane trees
on $[n]$ with $k$ leaves, $j$ degree one vertices and $i$ degree two vertices.
\end{proposition}

In this paper, we always let $\gamma_{n,i,j,k}$ be defined by~\eqref{Cnxyz}.
It follows from~\cite[eq.~(4.9)]{Chen21} that
\begin{equation*}
\gamma_{n,i,j,k}=3(i+1)\gamma_{n-1,i+1,j,k-1}+2(j+1)\gamma_{n-1,i-1,j+1,k-1}+k\gamma_{n-1,i,j-1,k},
\end{equation*}
with $\gamma_{1,0,0,1}=1$ and $\gamma_{1,i,j,k}=0$ if $k\neq 1$.
For $n=2,3,4$, the nonzero $\gamma_{n,i,j,k}$ are listed as follows:
$$\gamma_{2,0,1,1}=1,~\gamma_{3,1,0,2}=2,~\gamma_{3,0,2,1}=1,~\gamma_{4,0,0,3}=6,~\gamma_{4,1,1,2}=8,~\gamma_{4,0,3,1}=1.$$

In~\cite{Ma23}, we introduced Stirling permutation code, and provided numerous equidistribution results as applications.
The {\it trivariate ascent-plateau polynomial} is defined by
\begin{align*}
N_n(p,q,r)&=\sum_{\sigma\in\mqn}p^{\lap(\sigma)}q^{\eudd(\sigma)}r^{\rpd(\sigma)}.
\end{align*}
From~\cite[Theorem~21]{Ma23}, one see that
\begin{equation}\label{Nnxyz02}
N_n(p,q,r)=\sum_{i+2j+3k=2n+1}3^i\gamma_{n,i,j,k}(p+q+r)^{j}(pqr)^k,
\end{equation}
One may ask the following problem.
\begin{problem}\label{pr}
Why~\eqref{Cnxyz} and~\eqref{Nnxyz02} share the same coefficients?
\end{problem}
Let
\begin{equation}\label{Q60}
Q_{n}(x,y,z,p,q,r)=\sum_{\sigma\in\mqn}x^{\asc{(\sigma)}}y^{\plat{(\sigma)}}z^{\des(\sigma)}p^{\lap(\sigma)}
q^{\eudd(\sigma)}r^{\rpd(\sigma)}.
\end{equation}
A special case of Theorem~\ref{thm01} gives an answer to Problem~\ref{pr}:
\begin{equation}\label{Q6}
Q_{n}(x,y,z,p,q,r)=\sum_{i+2j+3k=2n+1}\gamma_{n,i,j,k}({x+y+z})^i({xyp+xzq+yzr})^j({xyzpqr})^{k}.
\end{equation}

This paper is organized as follows.
In Section~\ref{Section3}, we first prove Theorem~\ref{thm055}, and then we establish
a strong connection between signed permutations and Stirling permutations.
In Section~\ref{Section4}, using $\operatorname{SP}$-codes, we study a eight-variable polynomial $Q_{n}(x,y,z,p,q,r,s,t)$ as well as a seventeen-variable polynomial,
where $Q_{n}(x,y,z,p,q,r,1,1)=Q_{n}(x,y,z,p,q,r)$, i.e., the parameter $s$ marks the ascent-plateau-descent statistic and
$t$ marks the valley-valley pair statistic.
For the seventeen-variable polynomial, we find an
expansion formula with the same coefficients as in~\eqref{Cnxyz},~\eqref{Nnxyz02} and~\eqref{Q6}. Why the coefficients
$\gamma_{n,i,j,k}$ play such a crucial role in these expansions is somewhat mysterious.
From this paper, one can see that $\operatorname{SP}$-code is the key to clarify it.
\section{Eulerian polynomials and Stirling permutations}\label{Section3}
For an alphabet $A$, let $\mathbb{Q}[[A]]$ be the rational commutative ring of formal power
series in monomials formed from letters in $A$. Following Chen~\cite{Chen93}, a {\it context-free grammar} over
$A$ is a function $G: A\rightarrow \mathbb{Q}[[A]]$ that replaces each letter in $A$ by a formal function over $A$.
The formal derivative $D_G$ with respect to $G$ satisfies the derivation rules:
$$D_G(u+v)=D_G(u)+D_G(v),~D_G(uv)=D_G(u)v+uD_G(v).$$
In the theory of grammars, there are two methods for studying combinatorics: grammatical labeling and the change of grammar,
see~\cite{Chen2102,Chen23,Dumont96,Lin21,Ma1902,Ma23} for various applications.
\subsection{Proof of Theorem~\ref{thm055}}
\hspace*{\parindent}

\begin{lemma}\label{lemma1}
If $G=\{P\rightarrow PD+NA,~N\rightarrow PD+NA,~E\rightarrow (A+D)E,~A\rightarrow 2AD,~D\rightarrow 2AD\}$,
then for $n\geqslant 1$, we have
\begin{equation}\label{PN}
D_G^{n-1}(PE+NE)|_{P=A=E=1,~N=y,~D=xy}=\sum_{\pi\in\msn^B}x^{\operatorname{des}_A(\pi)}y^{\operatorname{des}_B(\pi)}.
\end{equation}
\end{lemma}

\begin{proof}
Given $\pi\in\msn^B$.
We first give a grammatical labeling of $\pi$ as follows:
\begin{itemize}
  \item [$(i)$] We put a superscript $P$ just before $\pi(1)$ if $\pi(1)>0$,
while we put a superscript $N$ just before $\pi(1)$ if $\pi(1)<0$;
  \item [$(ii)$] For $1\leqslant i\leqslant n-1$, we put a superscript
$A$ right after $\pi(i)$ if $\pi(i)<\pi(i+1)$, while we put a superscript $D$ right after $\pi(i)$ if $\pi(i)>\pi(i+1)$;
  \item [$(iii)$] Put a superscript $E$ at the end of $\pi$.
\end{itemize}
With this labeling, the weight of $\pi$ is defined as the product of its labels.
Note that the sum of weights of elements in $\ms_1^B$ is given by $PE+NE$.
In general, it is routine to verify that the insertion of $n$ or $\overline{n}$ corresponds to one substitution rule given by $G$.
For example, let $\pi=\overline{2}3\overline{1}54$. The labeling of $\pi $ is given by
$^N\overline{2}^A3^D\overline{1}^A5^D4^E$.
If we insert $6$ or $\overline{6}$ into $\pi$, we get
$$^P6^D\overline{2}^A3^D\overline{1}^A5^D4^E,~^N\overline{6}^A\overline{2}^A3^D\overline{1}^A5^D4^E,~^N\overline{2}^A6^D3^D\overline{1}^A5^D4^E,~^N\overline{2}^D\overline{6}^A3^D\overline{1}^A5^D4^E,$$
$$^N\overline{2}^A3^A6^D\overline{1}^A5^D4^E,~^N\overline{2}^A3^D\overline{6}^A\overline{1}^A5^D4^E,~^N\overline{2}^A3^D\overline{1}^A6^D5^D4^E,~^N\overline{2}^A3^D\overline{1}^D\overline{6}^A5^D4^E,$$
$$^N\overline{2}^A3^D\overline{1}^A5^A6^D4^E,~^N\overline{2}^A3^D\overline{1}^A5^D\overline{6}^A4^E,~^N\overline{2}^A3^D\overline{1}^A5^D4^A6^E,~^N\overline{2}^A3^D\overline{1}^A5^D4^D\overline{6}^E.$$
In conclusion, the action of the formal derivative $D_G$
on the set of weighted signed permutations in $\ms_{n}^B$ gives the set of weighted signed permutations in $\ms_{n+1}^B$. Substituting $P=A=E=1,~N=y$ and $D=xy$,
we get the desired result.
\end{proof}

\noindent{\bf A proof
Theorem~\ref{thm055}:}
\begin{proof}
Let $G$ be the grammar given by Lemma~\ref{lemma1}.
Consider a change of $G$. Setting $P+N=a,~b=PD+NA,~c=AD$ and $d=A+D$, we see that
$$D_G(a)=2b,~D_G(b)=2ac+bd,~D_G(c)=2cd,~D_G(d)=4c,~D_G(E)=dE.$$
Thus we get a new grammar:
$G'=\{a\rightarrow 2b,~b\rightarrow 2ac+bd,~c\rightarrow 2cd,~d\rightarrow 4c,~E\rightarrow dE\}$.
Note that $D_{G'}(aE)=adE+2bE$ and $D_{G'}^2(aE)=aE(d^2+8c)+bE(6d)$.
For $n\geqslant 2$, by induction, we find that there exist nonnegative integers $\xi({n,k})$ and $\zeta({n,k})$ such that
\begin{equation}\label{System-recu004}
D_{G'}^{n-1}(aE)=aE\sum_{k\geqslant 0}4^k\xi({n,k})c^kd^{n-1-2k}+bE\sum_{k\geqslant 0}4^k\zeta({n,k})c^kd^{n-2-2k}.
\end{equation}
Using $D_{G'}^{n}(aE)=D_{G'}\left(D_{G'}^{n-1}(aE)\right)$, it is easy to verify that
\begin{equation}\label{System-recu005}
\left\{
  \begin{array}{ll}
    \xi({n+1,k})=(1+2k)\xi(n,k)+(n-2k+1)\xi(n,k-1)+\frac{1}{2}\zeta(n,k-1), & \\
    \zeta({n+1,k})=2(1+k)\zeta({n,k})+(n-2k)\zeta({n,k-1})+2\xi(n,k), &
  \end{array}
\right.
\end{equation}
with $\xi({1,0})=1$ and $\xi({1,k})=0$ if $k\neq 0$ and $\zeta({1,k})=0$ for any $k$.
In~\eqref{System-recu004}, upon taking $a=1+y,~b=xy+y,~E=1,~c=xy$ and $d=1+xy$, we get the desired expansion formula.

Define
\begin{equation}\label{xi}
\xi_n(x)=\sum_{k\geqslant 0}\xi(n,k)x^k,~\zeta_n(x)=\sum_{k\geqslant 0}\zeta(n,k)x^{k},~f_n(x)=2\xi_n(x^2)+x\zeta_n(x^2).
\end{equation}
Let $T_n(x)=\sum_{\pi\in\msn}x^{\udrun(\pi)}=\sum_{k=1}^nT(n,k)x^k$.
The polynomials $T_n(x)$ satisfy the recursion
\begin{equation}\label{Rnx02}
T_{n+1}(x)=x(nx+1)T_{n}(x)+x\left(1-x^2\right)\frac{\mathrm{d}}{\mathrm{d}x}T_{n}(x),
\end{equation}
with $T_0(x)=1$ and $T_1(x)=x$ (see~\cite{Ma22,Stanley08}).
Based on empirical evidence,
we set
$T_n(x)=\frac{x}{2}\widehat{T}_n(x)$
for $n\geqslant 1$. It follows from~\eqref{Rnx02} that
$\widehat{T}_{n+1}(x)=(1+x+(n-1)x^2)\widehat{T}_{n}(x)+x\left(1-x^2\right)\frac{\mathrm{d}}{\mathrm{d}x}\widehat{T}_{n}(x)$.
From~\eqref{fnzeros}, we see that $\widehat{T}_n(x)$ satisfy the same recurrence relation and initial conditions as $f_n(x)$, so they agree.
Therefore, we deduce that $$T_n(x)=\frac{x}{2}f_n(x)=\frac{x}{2}\left(2\xi_n(x^2)+x\zeta_n(x^2)\right),$$
which yields $\xi(n,k)=T(n,2k+1)$ and $\zeta(n,k)=2T(n,2k+2)$. This completes the proof.
\end{proof}
\subsection{Real-rooted polynomials}
\hspace*{\parindent}

Let $\xi_n(x)$ and $\zeta_n(x)$ be defined by~\eqref{xi}.
Multiplying both sides of~\eqref{System-recu005} by $x^i$ and summing over all $i$, we arrive at the following result.
\begin{proposition}\label{recusys}
Let $\xi_1(x)=1$ and $\zeta_1(x)=0$. Then we have
$$\left\{
  \begin{array}{ll}
    \xi_{n+1}(x)=(1+(n-1)x)\xi_{n}(x)+2x(1-x)\frac{\mathrm{d}}{\mathrm{d}x}\xi_n(x)+\frac{x}{2}\zeta_n(x), & \\
    \zeta_{n+1}(x)=\left(2+(n-2)x\right)\zeta_n(x)+2x(1-x)\frac{\mathrm{d}}{\mathrm{d}x}\zeta_n(x)+2\xi_n(x). &
  \end{array}
\right.$$
\end{proposition}

In particular,
$\xi_2(x)=1,~\zeta_2(x)=2,~\xi_3(x)=1+2x,~\zeta_3(x)=6$. Moreover, we have $\xi_n(1)=\frac{n!}{2}$ and $\zeta_n(1)=n!$.
Let $\rz$ denote the set of real polynomials with only real zeros. Furthermore, denote
by $\rz(I)$ the set of such polynomials all of whose zeros are in the interval $I$.
\begin{theorem}\label{zeros}
Let $f_n(x)=2\xi_n(x^2)+x\zeta_n(x^2)$.
The polynomials $f_n(x)$ satisfy the recursion
\begin{equation}\label{fnzeros}
f_{n+1}(x)=(1+x+(n-1)x^2)f_n(x)+x(1-x^2)\frac{\mathrm{d}}{\mathrm{d}x}f_n(x),~f_1(x)=2.
\end{equation}
Then $f_n(x)\in\rz[-1,0)$ and $f_n(x)$ interlaces $f_{n+1}(x)$.
More precisely, $f_n(x)$ has $\lrf{\frac{n-1}{2}}$ simple zeros in the interval $(-1,0)$ and the zero $x=-1$ with the multiplicity $\lrf{\frac{n}{2}}$.
Moreover, both $\xi_n(x)$ and $\zeta_n(x)$ have only real negative simple zeros,
$\zeta_{2n}(x)$ alternates left of $\xi_{2n}(x)$ and $\zeta_{2n+1}(x)$ interlaces $\xi_{2n+1}(x)$.
\end{theorem}

In the sequel, we shall prove Theorem~\ref{zeros}.
Following~\cite{Wagner96},
we say that a polynomial $p(x)\in\R[x]$ is {\it standard} if its leading coefficient is positive.
Suppose that $p(x),q(x)\in\rz$.
The zeros of $p(x)$ are $\xi_1\leqslant\cdots\leqslant\xi_n$,
and that those of $q(x)$ are $\theta_1\leqslant\cdots\leqslant\theta_m$.
We say that $p(x)$ {\it interlaces} $q(x)$ if $\deg q(x)=1+\deg p(x)$ and the zeros of
$p(x)$ and $q(x)$ satisfy
$\theta_1\leqslant\xi_1\leqslant\theta_2\leqslant\cdots\leqslant\xi_n\leqslant\theta_{n+1}$.
We say that $p(x)$ {\it alternates left of} $q(x)$ if $\deg p(x)=\deg q(x)$
and the zeros of them satisfy
$
\xi_1\leqslant\theta_1\leqslant\xi_2\leqslant\cdots\leqslant\xi_n
\leqslant\theta_n$.
We use the notation $p(x)\prec q(x)$ for
either $p(x)$ interlaces $q(x)$ or $p(x)$ alternates left of $q(x)$.
A complex coefficients polynomial $p(x)$ is {\it weakly Hurwitz stable} if every zero of $p(x)$
is in the closed left half of the complex plane.
The following version of the
Hermite-Biehler theorem will be used in the proof of Theorem~\ref{zeros}.
\begin{HBtheorem}[{\cite[Theorem~3]{Wagner96}}]\label{Hermite}
Let $f(x)=f^E(x^2)+xf^O(x^2)$ be a standard polynomial with real coefficients.
Then $f(x)$ is weakly Hurwitz stable if and only if both $f^E(x)$ and $f^O(x)$ are standard,
have only nonpositive zeros, and $f^O(x)\prec f^E(x)$.
\end{HBtheorem}

\noindent{\bf A proof
Theorem~\ref{zeros}:}
\begin{proof}
Setting $a_n(x)=2\xi_{n}(x)$, it follows from Proposition~\ref{recusys} that
$$\left\{
  \begin{array}{ll}
    a_{n+1}(x)=(1+(n-1)x)a_{n}(x)+2x(1-x)\frac{\mathrm{d}}{\mathrm{d}x}a_n(x)+x\zeta_n(x), & \\
    \zeta_{n+1}(x)=\left(2+(n-2)x\right)\zeta_n(x)+2x(1-x)\frac{\mathrm{d}}{\mathrm{d}x}\zeta_n(x)+a_n(x).&
  \end{array}
\right.$$
Since $f_n(x)=2\xi_n(x^2)+x\zeta_n(x^2)=a_n(x^2)+x\zeta_n(x^2)$, then the recursion~\eqref{fnzeros} follows from~\cite[Theorem~3]{Ma22}.
When $n\geqslant 2$, it is clear that $f_n(0)=a_{n}(0)=a_1(0)=2$, $\deg (a_n(x))=\lrf{\frac{n-1}{2}}$ and $\deg (b_n(x))=\lrf{\frac{n-2}{2}}$.
By~\cite[Theorem~2]{Ma08}, we see that $f_n(x)$ interlaces $f_{n+1}(x)$, $f_n(x)$ has $\lrf{\frac{n-1}{2}}$ simple zeros in the interval $(-1,0)$ and the zero $x=-1$ with the multiplicity $\lrf{\frac{n}{2}}$.
It follows from Hermite-Biehler Theorem that both $a_n(x)$ and $\zeta_n(x)$ have only real negative simple zeros,
$\zeta_{2n}(x)$ alternates left of $a_{2n}(x)$ and $\zeta_{2n+1}(x)$ interlaces $a_{2n+1}(x)$. This completes the proof.
\end{proof}
\subsection{Relationship between signed permutations and Stirling permutations}
\hspace*{\parindent}

Let $\mqn^{(1)}=\{\sigma\in\mqn \mid \text{the two copies of $1$ are contiguous elements in $\sigma$}\}$. For example,
$$\mq_3^{(1)}=\{112233,~112332,~113322,~331122,~221133,~223311,~233211,~332211\}.$$
Denote by $\#C$ the cardinality of the set $C$.
Given $\sigma\in\mqn^{(1)}$.
Let us examine how to generate an element in $\mq_{n+1}^{(1)}$ by inserting two copies of $n+1$.
Note that there are $2n$ possibilities. Then $\#\mq_{n+1}^{(1)}=2n\#\mq_{n}^{(1)}=2^nn!$, since $\mq_1^{(1)}=\{11\}$.
Recall that $\#\msn^B=2^nn!$. It is natural to explore the relationship
between signed permutations and restricted Stirling permutation.

We need some notations.
For $\sigma\in\mqn$, a value $\sigma_i$ is called
\begin{itemize}
  \item an {\it ascent-plateau} if $\sigma_{i-1}<\sigma_{i}=\sigma_{i+1}$, where $2\leqslant i\leqslant 2n-1$;
  \item  a {\it left ascent-plateau} if $\sigma_{i-1}<\sigma_{i}=\sigma_{i+1}$, where $1\leqslant i\leqslant 2n-1$ and $\sigma_0=0$.
\end{itemize}
Let $\ap(\sigma)$ (resp.~$\lap(\sigma)$) be the number of ascent-plateaux (resp.~ left ascent-plateaux) in $\sigma$.
The {\it ascent-plateau polynomials} (also called {\it $1/2$-Eulerian polynomials}) $M_n(x)$~\cite{Ma19,Savage1202}
and the {\it left ascent-plateau polynomials} $\widetilde{M}_n(x)$~\cite{Liu21,Ma19} can be defined as follows:
\begin{equation*}
\begin{split}
M_n(x)&=\sum_{\sigma\in\mqn}x^{\ap(\sigma)},~
\widetilde{M}_n(x)=\sum_{\sigma\in\mqn}x^{\lap(\sigma)}.
\end{split}
\end{equation*}
From~\cite[Proposition~1]{MaYeh17}, we see that
\begin{equation*}\label{Convo01}
\begin{split}
2^nxA_n(x)&=\sum_{i=0}^n\binom{n}{i}\widetilde{M}_i(x)\widetilde{M}_{n-i}(x),~
B_n(x)=\sum_{i=0}^n\binom{n}{i}M_i(x)\widetilde{M}_{n-i}(x).
\end{split}
\end{equation*}

Let $\mq_{n+1}^{(0)}$ be the set of all Stirling permutations of the multiset $\{1,2^2,3^2,\ldots,n^2,(n+1)^2\}$, where only the element 1 appears one times.
Note that $\mqn^{(0)}\cong \mqn^{(1)}$. So $\#\mq_{n+1}^{(0)}=\#\mq_{n+1}^{(1)}=2^nn!$. For $\sigma\in\mq_{n+1}^{(0)}$,
we say that $\sigma_i$ is an {\it even value} if the first appearance of $\sigma_i$ occurs at an even position of $\sigma$, where $\sigma_i\in \{2,3,\ldots,n+1\}$.
Let $\even(\sigma)$ be the number of even indices in $\sigma$. For example, $\even(1\mathbf{2}2)=1,~\even (221)=0$ and $\even(2\mathbf{3}321)=1$.
\begin{theorem}\label{thm05}
For $n\geqslant 1$, we have
\begin{equation}\label{thm05result}
\sum_{\pi\in\msn^B}x^{\operatorname{des}_A(\pi)+1}y^{\operatorname{des}_B(\pi)}q^{\negg(\pi)}=\sum_{\sigma\in\mq_{n+1}^{(0)}}x^{\lap(\sigma)}y^{\ap(\sigma)}q^{\even(\sigma)}.
\end{equation}
\end{theorem}
\begin{proof}
Using the same labeling scheme for $\pi\in\msn^B$ as in the proof of Lemma~\ref{lemma1}, and attaching a weight $q$ to each negative element, one can easily verify that
 \begin{equation}\label{PN}
D_{G_1}^{n-1}(PE+qNE)|_{P=x,~A=E=1,~N=xy,~D=xy}=\sum_{\pi\in\msn^B}x^{\operatorname{des}_A(\pi)+1}y^{\operatorname{des}_B(\pi)}q^{\negg(\pi)},
\end{equation}
where the grammar $G_1$ is defined by
$G_1=\{P\rightarrow PD+qNA,~N\rightarrow PD+qNA,~E\rightarrow (A+qD)E,~A\rightarrow (1+q)AD,~D\rightarrow (1+q)AD\}$.

Given $\sigma\in\mq_{n}^{(0)}$.
We now define a labeling scheme for the Stirling permutation $\sigma$:
\begin{itemize}
  \item [$(i)$] If $\sigma_1=\sigma_2$, then $\sigma_1$ is a left ascent-plateau, and we label the two positions just before and right after $\sigma_1$ by a subscript label $\alpha$.
For any other ascent-plateau $\sigma_i$, i.e., $i\geqslant 2$ and $\sigma_{i-1}<\sigma_i=\sigma_{i+1}$,
we label the two positions just before and right after $\sigma_i$ by a label $\beta$;
  \item [$(ii)$] If $\sigma_1<\sigma_2$, then we use the superscript $\gamma$ to mark the first position (just before $\sigma_1$) and the last position (at the end of $\sigma$), and denoted by $\overbrace{\sigma_1\sigma_2\cdots \sigma_{2n}}^{\gamma}$;
  \item [$(iii)$] For any other element, we label it by a subscript $w$;
  \item [$(iv)$] We attach a superscript label $q$ to every even value.
\end{itemize}
As illustrations, the labeled elements in $\mq_{2}^{(0)}$ can be listed as follows:
\begin{equation*}
\overbrace{1\underbrace{2^q}_{\beta}2}^{\gamma},~\underbrace{2}_{\alpha}2_{w}1_{w}.
\end{equation*}
Let us examine how to generate elements in $\mq_{3}^{(0)}$ by inserting the two copies of $3$:
\begin{equation}\label{even}
\overbrace{1\underbrace{2^q}_{\beta}2}^{\gamma} \left\{
                                                             \begin{array}{ll}
                                                               \underbrace{3}_{\alpha}3_{w}1\underbrace{2^q}_{\beta}2_{w}, & \\
                                                             \overbrace{1\underbrace{3^q}_{\beta}3_{w}2^q_{w}2}^{\gamma}, &  \\
                                                                \overbrace{1_{w}2^q\underbrace{3}_{\beta}3_{w}2}^{\gamma}, &  \\
                                                              \overbrace{ 1\underbrace{2^q}_{\beta}2\underbrace{3^q}_{\beta}3}^{\gamma}, &
                                                             \end{array}
                                                           \right.
~\underbrace{2}_{\alpha}2_{w}1_{w}\rightarrow \left\{
                                                             \begin{array}{ll}
                                                               \underbrace{3}_{\alpha}3_{w}2_{w}2_{w}1_{w}, & \\
                                                             \overbrace{ 2\underbrace{3^q}_{\beta}3_{w}2_{w}1}^{\gamma}, &  \\
                                                               \underbrace{2}_{\alpha}2\underbrace{3}_{\beta}3_{w}1_{w}, &  \\
                                                               \underbrace{2}_{\alpha}2_{w}1\underbrace{3^q}_{\beta}3_{w}. &
                                                             \end{array}
                                                           \right.
\end{equation}

Note that the labels $w$ always appear even times. We can read the labels $w$ two by two from left to right. So we set $W=w^2$, where
one $w$ marks an element in odd position and the other marks a nearest element (from left to right) in even position.
Note that the sum of weights of elements in $\mq_{2}^{(0)}$ is given by
$\alpha w^2+q\beta\gamma=\alpha W+q\beta\gamma$. By induction, as illustrated by~\eqref{even},
we see that if
$$G_2=\{\alpha\rightarrow \alpha W+q\beta\gamma,~\gamma\rightarrow\alpha W+q\beta\gamma,~\beta\rightarrow (1+q)\beta W,~W\rightarrow (1+q)\beta W\},$$
then
\begin{equation}
D_{G_2}^{n-1}(\alpha W+q\beta\gamma)|_{\alpha=x,~\gamma=W=1,~\beta=xy}=\sum_{\sigma\in\mq_{n+1}^{(0)}}x^{\lap(\sigma)}y^{\ap(\sigma)}q^{\even(\sigma)}.
\end{equation}

Note that
$$D_{G_1}(PE+qNE)=PE(A+D+2qD)+qNE(2A+qA+qD),$$
$$D_{G_2}(\alpha W+q\beta\gamma)=\alpha W(W+\beta+2q\beta)+q\beta\gamma(2W+qW+q\beta).$$
When $n\geqslant 1$, by induction, it is routine to check that
$$D_{G_1}^n(PE+qNE)=PEf_{n}(A,D;q)+qNEg_{n}(A,D;q),$$
$$D_{G_2}^n(\alpha W+q\beta\gamma)=\alpha W f_{n}(W,\beta;q)+q\beta\gamma g_{n}(W,\beta;q),$$
where $f_n(x,y;q)$ and $g_n(x,y;q)$ satisfy the recurrence system:
$$\left\{
  \begin{array}{ll}
  f_{n+1}(x,y;q)=(x+y+qy)f_n(x,y;q)+D_{G_3}\left(f_n(x,y;q)\right)+qyg_n(x,y;q), & \\
  g_{n+1}(x,y;q)=(x+qx+qy)g_n(x,y;q)+D_{G_3}(g_n(x,y;q))+xf_n(x,y;q), &
  \end{array}
\right.$$
with $f_1(x,y;q)=g_1(x,y;q)=1$ and $G_3=\{x\rightarrow (1+q)xy,~y\rightarrow (1+q)xy\}$.
Therefore, if we make the substitutions: $PE\Leftrightarrow \alpha W,~NE\Leftrightarrow \beta\gamma,~A\Leftrightarrow W,~D\Leftrightarrow \beta$,
we obtain
$$D_{G_1}^{n-1}(PE+qNE)|_{P=x,~A=E=1,~N=xy,~D=xy}=D_{G_2}^{n-1}(\alpha\gamma+q\beta\delta)|_{\alpha=x,~\gamma=W=1,~\beta=xy},$$
and so we arrive at~\eqref{thm05result}. This completes the proof.
\end{proof}

A permutation is called a {\it derangement} if is has no fixed points. Let $\mdd_n$ be the set of all derangements in $\msn$.
The {\it derangement polynomials} are defined by $d_n(x)=\sum_{\pi\in\mdd_n}x^{\exc(\pi)}$.
Combining~\eqref{thm05result} and~\cite[Eq.~(14)]{Brenti94}, we discover the following result.
\begin{corollary}
We have $$\sum_{\sigma\in\mq_{n+1}^{(0)}}y^{\ap(\sigma)}(-1)^{\even(\sigma)}=(1-y)^n,~\sum_{\sigma\in\mq_{n+1}^{(0)}}y^{\ap(\sigma)-\even(\sigma)}(-1)^{\even(\sigma)}=y\left(\frac{y-1}{y}\right)^nd_n(y).$$
\end{corollary}
The {\it type $D$ Coxeter group} $\msn^D$ is the subgroup of
$\msn^B$ consisting of signed permutations with an even number of negative entries.
It follows from Theorem~\ref{thm05} that $$\#\{\sigma\in \mq_{n+1}^{(0)}\mid \text{$\even(\sigma)$ is odd} \}=\#\{\sigma\in \mq_{n+1}^{(0)}\mid \text{$\even(\sigma)$ is even} \}=\#\msn^D=2^{n-1}n!,$$ since $\#\msn^B=2^nn!$.
Let $\operatorname{des}_D(\pi)=\#\{i\in[n]\mid \pi(i-1)>\pi({i})\}$, where $\pi(0)=-\pi(2)$.
The {\it type $D$ Eulerian polynomial} is defined by
$${D}_n(x)=\sum_{\pi\in \msn^D}x^{\operatorname{des}_D(\pi)}.$$
Stembridge~\cite[Lemma 9.1]{Stembridge94} obtained that
$D_n(x)=B_n(x)-n2^{n-1}xA_{n-1}(x)$
for $n\geqslant 2$. We end this section by posing two problems.
\begin{problem}
Could we find a combinatorial interpretation of the type $D$ Eulerian polynomial $D_n(x)$ in terms of Stirling permutations in the set $\{\sigma\in \mq_{n+1}^{(0)}\mid \text{$\even(\sigma)$ is even} \}$?
\end{problem}
\begin{problem}
How to relate $\mq_{n}^{(0)}$ and $\{\sigma\in \mq_{n}^{(0)}\mid \text{$\even(\sigma)$ is even} \}$ to group operations?
\end{problem}
\section{Eight-variable and seventeen-variable polynomials}\label{Section4}
\subsection{Definitions, notation and preliminary results}
\hspace*{\parindent}

In equivalent forms, Dumont~\cite{Dumont80}, Haglund-Visontai~\cite{Haglund12}, Chen-Hao-Yang~\cite{Chen2102} and Ma-Ma-Yeh~\cite{Ma1902} all showed that
$$D_G^n(x)=C_n(x,y,z),$$
where $G=\{x \rightarrow xyz,~y\rightarrow xyz,~z\rightarrow xyz\}$.
By the change of grammar
$u=x+y+z,~v=xy+yz+zx$ and $w=xyz$,
it is clear that
$D_{G}(u)=3w,~D_{{G}}(v)=2uw,~D_{G}(w)=vw$. So we get a grammar
\begin{equation}\label{G8def}
H=\{w\rightarrow vw, u\rightarrow 3w, v\rightarrow 2uw\}.
\end{equation}
For any $n\geqslant 1$, Chen-Fu~\cite{Chen21} discovered that
\begin{equation}\label{Chen22}
C_n(x,y,z)=D_G^n(x)=D_H^{n-1}(w)=\sum_{i+2j+3k=2n+1}\gamma_{n,i,j,k}u^iv^jw^k.
\end{equation}
Substituting $u\rightarrow x+y+z,~v\rightarrow xy+yz+zx$ and $w\rightarrow xyz$, one can immediately obtain~\eqref{Cnxyz}.

A {\it rooted tree} of order $n$ with the vertices labelled $1,2,\ldots,n$, is an increasing tree if the
node labelled $1$ is distinguished as the root, and the labels along
any path from the root are increasing.
\begin{definition}
A {\it ternary increasing tree} of size $n$ is an increasing plane tree with $3n+1$ nodes in which each
interior node has a label and three ordered children (a left child,
a middle child and a right child), and exterior nodes have no children and no labels.
\end{definition}

Let $\mtn_n$ denote the set of ternary increasing trees of size $n$, see Figure~\ref{Fig01} for instance. For any $T\in\mtn_n$,
it is clear that $T$ has exactly $2n+1$ exterior nodes.
Let $\lends(T)$ (resp.~$\mends(T)$,~$\rends(T)$) be the number of exterior left nodes (resp.~exterior middle nodes, exterior right nodes) in $T$.
Using a recurrence relation that is equivalent to~\eqref{Dumont80}, Dumont~\cite[Proposition~1]{Dumont80} found that
\begin{equation}\label{QnTn}
\sum_{\sigma\in\mqn}x^{\asc(\sigma)}y^{\plat(\sigma)}z^{\des(\sigma)}=\sum_{T\in\mtn_n}x^{\lends(T)}y^{\mends(T)}z^{\rends(T)}.
\end{equation}
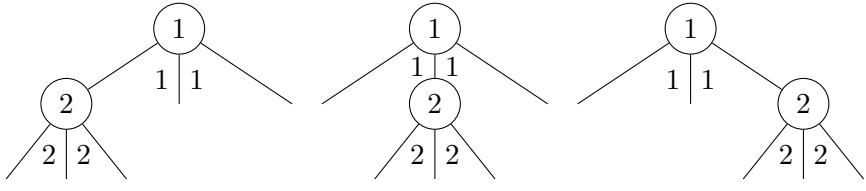
\begin{figure}[tbp]
\begin{center}
\begin{tikzpicture} 
\node [circle,draw] {1}
    child {node [circle,draw] {2}
    		child { edge from parent }
    		child { edge from parent node [left] {2} node [right] {2}}
    		child { edge from parent }
    }
    child { edge from parent node [left] {1} node [right] {1} }
    child { edge from parent };
\end{tikzpicture}\quad 
\begin{tikzpicture}
\node [circle,draw] {1}
    child { edge from parent }
    child { node [circle,draw] {2}     		
    		child { edge from parent }
    		child { edge from parent node [left] {2} node [right] {2}}
    		child { edge from parent }
    		edge from parent node[left] {1} node[right] {1}
    }
    child { edge from parent };
\end{tikzpicture}\quad %
\begin{tikzpicture}
\node [circle,draw] {1}
    child { edge from parent }
    child { edge from parent node[left] {1} node[right] {1}}
    child { node [circle,draw] {2}     		
    		child { edge from parent }
    		child { edge from parent node [left] {2} node [right] {2}}
    		child { edge from parent }
    		edge from parent
    };
\end{tikzpicture}
\end{center}
\caption{The ternary increasing trees of order 2 encoded by $2211,1221,1122$,
and their $\operatorname{SP}$-codes are given by $((0,0),(1,1)),((0,0)(1,2))$ and $((0,0)(1,3))$, respectively .}
\label{Fig01}
\end{figure}

A bijective proof of~\eqref{QnTn} can be found in the proof of~\cite[Theorem 1]{Janson11}. For the convenience of the reader,
we provide a brief description of it. Given $T\in\mtn_n$. Between the $3$ edges of $T$ going out from a node labelled $v$, we place $2$ integers $v$.
Now we perform the depth-first traversal and encode the tree $T$ by recording the sequence of the labels visited as we traverse around $T$.
The encoded sequence, denoted as $\phi(T)$, is a Stirling permutation.
Given $\sigma\in\mqn$. We proceed recursively by decomposing $\sigma$ as $u_11u_21u_3$, where the $u_i$'s are again Stirling permutations.
The smallest label in each $u_i$ is attached to the root node labelled $1$. One can recursively apply this procedure to each $u_i$ to obtain the tree representation, and
$\phi^{-1}(\sigma)$ is a ternary increasing tree. Using $\phi$, one can immediately find that the identity~\eqref{QnTn} holds.

 For $\sigma\in\mqn$, let
\begin{align*}
\operatorname{Dplat}(\sigma)&=\{\sigma_i\mid \sigma_{i-1}>\sigma_i=\sigma_{i+1}\},~
\operatorname{Dasc}(\sigma)=\{\sigma_i\mid \sigma_{i-1}<\sigma_{i}<\sigma_{i+1}\},\\
\operatorname{Dd}(\sigma)&=\{\sigma_i\mid \sigma_{i-1}>\sigma_i=\sigma_{j}>\sigma_{j+1},~i<j\},\\
\operatorname{Uu}(\sigma)&=\{\sigma_i\mid \sigma_{i-1}<\sigma_i=\sigma_{j}<\sigma_{j+1},~i<j\},\\
\operatorname{Ddes}(\sigma)&=\{\sigma_i\mid \sigma_{i-1}>\sigma_i>\sigma_{i+1}\},~
\operatorname{Pasc}(\sigma)=\{\sigma_i\mid \sigma_{i-1}=\sigma_i<\sigma_{i+1}\},\\
\operatorname{Apap}(\sigma)&=\{\sigma_i\mid \sigma_{i}<\sigma_{i+1}=\sigma_{i+2}~\&~\sigma_{j}<\sigma_{j+1}=\sigma_{j+2}~\&~\sigma_i=\sigma_j,~i<j\},\\
\operatorname{Dpa}(\sigma)&=\{\sigma_{i}\mid \sigma_{i-1}>\sigma_{i}=\sigma_{i+1}<\sigma_{i+2}\},\\
\operatorname{Pdpd}(\sigma)&=\{\sigma_i\mid \sigma_{i-2}=\sigma_{i-1}>\sigma_{i}~\&~\sigma_{j-2}=\sigma_{j-1}>\sigma_{j}~\&~\sigma_i=\sigma_j,~i<j\}.
\end{align*}
\begin{table}[h!]
 \begin{center}
    \begin{tabular}{l|c} 
      \textbf{Statistics on Stirling permutation} & \textbf{Statistics on $\operatorname{SP}$-code}\\
      \hline
      $\Asc$ (ascent) & $[n]-\{a_i\mid (a_i,1)\in C_n\}$ \\
      $\Plat$ (plateau) & $[n]-\{a_i\mid (a_i,2)\in C_n\}$ \\
      $\Des$ (descent)& $[n]-\{a_i\mid (a_i,3)\in C_n\}$ \\
    $\Lap$ (left ascent-plateau)& $[n]-\{a_i\mid (a_i,1)~\text{or}~ (a_i,2)\in C_n\}$ \\
    $\Rpd$ (right plateau-descent)& $[n]-\{a_i\mid (a_i,2)~\text{or}~ (a_i,3)\in C_n\}$ \\
   $\Eudd$ (exterior up-down pair)& $[n]-\{a_i\mid (a_i,1)~\text{or}~ (a_i,3)\in C_n\}$ \\
    $\Dasc$ (double ascent)& $\{a_i\mid (a_i,1)\notin C_n~\&~ (a_i,2)\in C_n\}$ \\
    $\Dplat$ (descent-plateau)& $\{a_i\mid (a_i,1)\in C_n~\&~ (a_i,2)\notin C_n\}$ \\
   $\Ddes$ (double descent)& $\{a_i\mid (a_i,2)\in C_n~\&~ (a_i,3)\notin C_n\}$ \\
  $\Pasc$ (plateau-ascent)& $\{a_i\mid (a_i,2)\notin C_n~\&~ (a_i,3)\in C_n\}$ \\
  $\operatorname{Uu}$ (up-up pair)& $\{a_i\mid (a_i,1)\notin C_n~\&~ (a_i,3)\in C_n\}$\\
  $\operatorname{Dd}$ (down-down pair)& $\{a_i\mid (a_i,1)\in C_n~\&~ (a_i,3)\notin C_n\}$\\
   $\operatorname{Apd}$ (ascent-plateau-descent)& $\{a_i\mid (a_i,1)\notin C_n~\&~ (a_i,2)\notin C_n~\&~ (a_i,3)\notin C_n\}$\\
  $\operatorname{Vv}$ (valley-valley pair)& $\{a_i\mid (a_i,1)\in C_n~\&~(a_i,2)\in C_n~\&~ (a_i,3)\in C_n\}$\\
    $\operatorname{Apap}$ (ascent-plateau pair)& $\{a_i\mid (a_i,1)\notin C_n~\&~(a_i,2)\in C_n~\&~ (a_i,3)\in C_n\}$\\
      $\operatorname{Dpa}$ (descent-plateau-ascent)& $\{a_i\mid (a_i,1)\in C_n~\&~(a_i,2)\notin C_n~\&~ (a_i,3)\in C_n\}$\\
        $\operatorname{Pdpd}$ (plateau-descent pair)& $\{a_i\mid (a_i,1)\in C_n~\&~(a_i,2)\in C_n~\&~ (a_i,3)\notin C_n\}$
    \end{tabular}
  \end{center}
\caption{The correspondences of statistics on Stirling permutations and $\operatorname{SP}$-codes}
\label{Table1}
\end{table}

denote the sets of descent-plateaux, double ascents, down-down pairs, up-up pairs, double descents, plateau-ascents,
ascent-plateau pairs, descent-plateau-ascents and plateau-descent pairs of $\sigma$, respectively.
Let $\dplat(\sigma)$ (resp.~$\dasc(\sigma)$, $\operatorname{dd}(\sigma)$, $\operatorname{uu}(\sigma)$, $\ddes(\sigma)$, $\pasc(\sigma)$, $\operatorname{apap}(\sigma)$ $\operatorname{dpa}(\sigma)$, $\operatorname{pdpd}(\sigma)$) denote
the number of descent-plateaux (resp.~double ascents, down-down pairs, up-up pairs, double descents, plateau-ascents, ascent-plateau pairs, descent-plateau-ascents, plateau-descent pairs) in $\sigma$. It should be noted that the statistics $\operatorname{apap}$, $\operatorname{dpa}$, $\operatorname{pdpd}$ are all new statistics.

In the sequel, we give a summary of Stirling permutation codes.
Following~\cite[p.~10]{Ma23}, any ternary increasing tree of size $n$
can be built up from the root $1$ by successively adding nodes $2,3,\ldots,n$.
Clearly, node $2$ is a child of the root $1$ and the root $1$ has at most three children.
For $2\leqslant i\leqslant n$,
when node $i$ is inserted, we distinguish three cases:
\begin{enumerate}
\item [$(c_1)$] if it is the left child of a node $v\in [i-1]$, then the node $i$ is coded as $[v,1]$;
\item [$(c_2)$] if it is the middle child of a node $v\in [i-1]$, then the node $i$ is coded as $[v,2]$;
\item [$(c_3)$] if it is the right child of a node $v\in [i-1]$, then the node $i$ is coded as $[v,3]$.
 \end{enumerate}
Thus the node $i$ is coded as a $2$-tuple $(a_{i-1},b_{i-1})$, where $1\leqslant a_{i-1}\leqslant i-1$, $1\leqslant b_{i-1}\leqslant 3$ and $(a_i,b_i)\neq(a_j,b_j)$ for all $1\leqslant i<j\leqslant n-1$.
For convenience, we name this build-tree code as Stirling permutation code. By convention, the root $1$ is coded as $(0,0)$.

\begin{definition}[{\cite[Definition~9]{Ma23}}]\label{def-BSP}
A $2$-tuples sequence $C_n=((0,0),(a_1,b_1),(a_2,b_2)\ldots,(a_{n-1},b_{n-1}))$ of length $n$ is called a Stirling permutation code ($\operatorname{SP}$-code for short)
if $1\leqslant a_i\leqslant i$, $1\leqslant b_i\leqslant 3$ and $(a_i,b_i)\neq(a_j,b_j)$ for all $1\leqslant i<j\leqslant n-1$.
\end{definition}
Let $\operatorname{CQ}_n$ be the set of $\operatorname{SP}$-codes of length $n$.
In particular, we have $$\operatorname{CQ}_1=\{(0,0)\},~\operatorname{CQ}_2=\{(0,0)(1,1),~(0,0)(1,2),~(0,0)(1,3)\}.$$

We now describe a bijection $\Gamma$ between $\mqn$ and $\operatorname{CQ}_n$.
There are three cases to obtain an element of $\mqn$ from an element $\sigma\in\mq_{n-1}$ by
inserting the two copies of $n$ between $\sigma_i$ and $\sigma_{i+1}$: $\sigma_i<\sigma_{i+1},~\sigma_i=\sigma_{i+1}$ or $\sigma_i>\sigma_{i+1}$.
Set $\Gamma(11)=(0,0)$. When $n\geqslant 2$, the bijection $\Gamma: \mqn\rightarrow \operatorname{CQ}_n$ can be defined as follows:
\begin{enumerate}
\item [$(c_1)$] $\sigma_i<\sigma_{i+1}$ if and only if $(a_{n-1},b_{n-1})=(\sigma_{i+1},1)$;
\item [$(c_2)$] $\sigma_i=\sigma_{i+1}$ if and only if $(a_{n-1},b_{n-1})=(\sigma_{i+1},2)$;
\item [$(c_3)$] $\sigma_i>\sigma_{i+1}$ if and only if $(a_{n-1},b_{n-1})=(\sigma_{i},3)$.
 \end{enumerate}

As discussed in~\cite{Ma23},
combing the bijections $\phi$ (from Stirling permutations to ternary increasing trees) and $\Gamma$ (from ternary increasing trees to Stirling permutation codes), we obtain Table~\ref{Table1}, which contains the correspondences of set valued statistics on Stirling permutations and $\operatorname{SP}$-codes.
It should be noted that the last four are new correspondences.
\subsection{Main results}
\hspace*{\parindent}

In order to study the six-variable polynomials defined by~\eqref{Q60},
we find that it would be necessary to introduce the following eight-variable polynomials:
$$Q_{n}(x,y,z,p,q,r,s,t)=\sum_{\sigma\in\mqn}x^{\asc{(\sigma)}}y^{\plat{(\sigma)}}z^{\des(\sigma)}p^{\lap(\sigma)}
q^{\eudd(\sigma)}r^{\rpd(\sigma)}s^{\operatorname{apd}(\sigma)}t^{\operatorname{vv}(\sigma)}.$$
where $\operatorname{apd}(\sigma)$ and $\operatorname{vv}(\sigma)$ are the numbers of ascent-plateau-descents and valley-valley pairs of $\sigma$, respectively.
In particular, $Q_1=xyzpqrs$ and $Q_2=xyzpqrs(xyp+xzq+yzr)$. For convenience, we set $Q_n:=Q_{n}(x,y,z,p,q,r,s,t)$.
We are now ready to answer Problem~\ref{pr}.
\begin{theorem}\label{thm01}
For any $n\geqslant 1$, we have the following decomposition
\begin{equation*}
Q_n(x,y,z,p,q,r,s,t)=t^{n}\sum_{i+2j+3k=2n+1}\gamma_{n,i,j,k}\left(\frac{x+y+z}{t}\right)^i\left(\frac{xyp+xzq+yzr}{t}\right)^j\left(\frac{xyzpqrs}{t}\right)^{k}.
\end{equation*}
\end{theorem}
We next relate the polynomial $Q_n$ to a five-variable polynomial.
\begin{theorem}\label{thm02}
The polynomial $Q_n:=Q_n(x,y,z,p,q,r,s,t)$ can be expanded as follows:
$$Q_n=\left(\frac{x+y+z}{3}\right)^nQ_n\left(1,1,1,\frac{3xyp}{x+y+z},\frac{3xzq}{x+y+z},\frac{3yzr}{x+y+z},\frac{s(x+y+z)^2}{9xyz},\frac{3t}{x+y+z}\right).$$
Since $C_n(x,y,z)=Q_n(x,y,z,1,1,1,1,1)$, we obtain
\begin{equation}\label{thm02cor2}
{C_n(x,y,z)}=\sum_{\sigma\in\mqn}(xy)^{\lap(\sigma)}(xz)^{\eudd(\sigma)}
(yz)^{\rpd(\sigma)}(xyz)^{-\operatorname{apd}(\sigma)}\left(\frac{x+y+z}{3}\right)^{\alpha_n(\sigma)},
\end{equation}
where $\alpha_n(\sigma)=n+2\operatorname{apd}(\sigma)-\lap(\sigma)-\eudd(\sigma)-\rpd(\sigma)-\operatorname{vv}(\sigma)$.
\end{theorem}
It should be noted that~\eqref{thm02cor2} implies that the following two results are equivalent:
\begin{itemize}
  \item the triple statistic $(\asc,\plat,\des)$ is a symmetric distribution over $\mqn$;
  \item the triple statistic $(\lap,\eudd,\rpd)$ is a symmetric distribution over $\mqn$.
\end{itemize}

Define
$$M_n(\beta_1,\beta_3,\beta_5)=\sum_{\sigma\in\mqn}\beta_1^{\dplat(\sigma)}\beta_4^{\operatorname{uu}(\sigma)}\beta_5^{\ddes(\sigma)},$$
$$P_n(\alpha_1,\alpha_2,\alpha_3)=\sum_{\sigma\in\mqn}\alpha_1^{\operatorname{apap}(\sigma)}\alpha_2^{\operatorname{dpa}(\sigma)}\alpha_3^{\operatorname{pdpd}(\sigma)},$$
$$E_n(\beta_1,\beta_2,\beta_3,\beta_4,\beta_5,\beta_6)=\sum_{\sigma\in\mqn}\beta_1^{\dplat(\sigma)}\beta_2^{\dasc(\sigma)}\beta_3^{\operatorname{dd}(\sigma)}\beta_4^{\operatorname{uu}(\sigma)}
\beta_5^{\ddes(\sigma)}\beta_6^{\pasc(\sigma)},$$
$$\alpha(\sigma)=\alpha_1^{\operatorname{apap}(\sigma)}\alpha_2^{\operatorname{dpa}(\sigma)}\alpha_3^{\operatorname{pdpd}(\sigma)},$$
$$\beta(\sigma)=\beta_1^{\dplat(\sigma)}\beta_2^{\dasc(\sigma)}\beta_3^{\operatorname{dd}(\sigma)}\beta_4^{\operatorname{uu}(\sigma)}
\beta_5^{\ddes(\sigma)}\beta_6^{\pasc(\sigma)}.$$
Let $F_n$ denote the following seventeen-variable polynomials:
$$\sum_{\sigma\in\mqn}\alpha(\sigma)\beta(\sigma)x^{\asc{(\sigma)}}y^{\plat{(\sigma)}}z^{\des(\sigma)}p^{\lap(\sigma)}
q^{\eudd(\sigma)}r^{\rpd(\sigma)}s^{\operatorname{apd}(\sigma)}t^{\operatorname{vv}(\sigma)}.$$
We can now present a generalization of Theorem~\ref{thm01}.
\begin{theorem}\label{thm03}
For any $n\geqslant 1$, the seventeen-variable polynomial $F_n$ has the expansion formula:
\begin{equation}\label{Fn-expansion}
F_n=t^n\sum_{i+2j+3k=2n+1}\gamma_{n,i,j,k}\left(\frac{\delta_2}{t}\right)^i\left(\frac{\delta_1}{t}\right)^j\left(\frac{\delta}{t}\right)^k,
\end{equation}
where $\delta=xyzpqrs,~\delta_1=\beta_4\beta_6xyp+\beta_2\beta_5xzq+\beta_1\beta_3yzr$ and $\delta_2=\alpha_1\beta_2\beta_4x+\alpha_2\beta_1\beta_6y+\alpha_3\beta_3\beta_5z$.
\end{theorem}

A special case of~\eqref{Fn-expansion} is given as follows.
\begin{corollary}
The trivariate polynomial $M_n(\beta_1,\beta_3,\beta_5)$ is $e$-positive. More precisely,
$$M_n(\beta_1,\beta_3,\beta_5)=\sum_{i+2j+3k=2n+1}\gamma_{n,i,j,k}(\beta_1+\beta_4+\beta_5)^{i+j}.$$
\end{corollary}

As a unified extension of $N_n(p,q,r)$ and $P_n(\alpha_1,\alpha_2,\alpha_3)$, consider the six-variable polynomials
$$NP_n(p,q,r,\alpha_1,\alpha_2,\alpha_3)=\sum_{\sigma\in\mqn}p^{\lap(\sigma)}q^{\eudd(\sigma)}r^{\rpd(\sigma)}
\alpha_1^{\operatorname{apap}(\sigma)}\alpha_2^{\operatorname{dpa}(\sigma)}\alpha_3^{\operatorname{pdpd}(\sigma)}.$$
\begin{corollary}
The six-variable polynomials $NP_n(p,q,r,\alpha_1,\alpha_2,\alpha_3)$ can be expanded as follows:
$$NP_n(p,q,r,\alpha_1,\alpha_2,\alpha_3)=\sum_{i+2j+3k=2n+1}\gamma_{n,i,j,k}(\alpha_1+\alpha_2+\alpha_3)^i(p+q+r)^{j}(pqr)^k.$$
When $p=q=r=1$, we see that the polynomial $P_n(\alpha_1,\alpha_2,\alpha_3)$ is $e$-positive, i.e.,
$$P_n(\alpha_1,\alpha_2,\alpha_3)=\sum_{i+2j+3k=2n+1}\gamma_{n,i,j,k}3^j(\alpha_1+\alpha_2+\alpha_3)^i.$$
\end{corollary}

Note that $$E_n(\beta_1,\beta_2,\beta_3,\beta_4,\beta_5,\beta_6)=\sum_{i+2j+3k=2n+1}\gamma_{n,i,j,k}
\left({\beta_2\beta_4+\beta_1\beta_6+\beta_3\beta_5}\right)^i
\left({\beta_4\beta_6+\beta_2\beta_5+\beta_1\beta_3}\right)^j.$$
\begin{corollary}
We have $E_n(x,y,1,1,1,1)=E_n(1,1,x,y,1,1)=E_n(1,1,1,1,x,y)$ and
the polynomial $E_n(x,y,1,1,1,1)$ is $e$-positive.
\end{corollary}
\subsection{Proof of Theorem~\ref{thm01}}
\hspace*{\parindent}

A {\it simplified ternary increasing tree} is a ternary increasing tree with no exterior nodes.
In Figure~\ref{Fig002}, we give the simplified ternary increasing trees of order $2$,
where the left figure represents the three different figures in the right.
The {\it degree} of a vertex in a simplified ternary increasing tree is meant to be the number of its children.
\begin{figure}
\begin{center}
\begin{tikzpicture}
[emph/.style={edge from parent/.style={snakeline,draw}}]
\node (1) [circle,draw] {1}
    child [emph] {node (2) [circle,draw] {2}};
\path (1.east) node[above right] {[$P_1$]};
\path (2.east) node[above right] {[$P$]};
\end{tikzpicture}
=
\begin{tikzpicture}
[level 1/.style = {sibling distance = .7cm},
NONE/.style={edge from parent/.style={draw=none}}]
\node [circle,draw] {1}
    child {node [circle,draw] {2}}
    child [NONE] {}
    child [NONE] {};
\end{tikzpicture}
or
\begin{tikzpicture}
[level 1/.style = {sibling distance = .7cm},
NONE/.style={edge from parent/.style={draw=none}}]
\node [circle,draw] {1}
    child [NONE] {}
    child {node [circle,draw] {2}}
    child [NONE] {};
\end{tikzpicture}
or
\begin{tikzpicture}
[level 1/.style = {sibling distance = .7cm},
NONE/.style={edge from parent/.style={draw=none}}]
\node [circle,draw] {1}
    child [NONE] {}
    child [NONE] {}
    child {node [circle,draw] {2}};
\end{tikzpicture}
\caption{Simplified ternary increasing trees,~$Q_2=xyzpqrs(xyp+xzq+yzr)=PP_1$.}
\label{Fig002}
\end{center}
\end{figure}
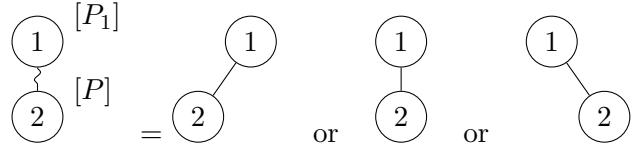

The weight $W_1$ of $\sigma\in\mqn$ is defined by
$$W_1(\sigma)=x^{\asc{(\sigma)}}y^{\plat{(\sigma)}}z^{\des(\sigma)}p^{\lap\sigma)}q^{\eudd(\sigma)}r^{\rpd(\sigma)}s^{\operatorname{apd}(\sigma)}t^{\operatorname{vv}(\sigma)}.$$
Using Table~\ref{Table1}, one can get the corresponding weight of $C_n\in\operatorname{CQ}_n$, and we use $W_2(C_n)$ to denote it, see~\eqref{subsection34} for a more general case. In other words,
if $\Gamma(\sigma)=C_n$, then $W_1(\sigma)=W_2(C_n)$.
In the following, we always set $P:=xyzpqrs,~P_1:=xyp+xzq+yzr$ and $P_2:=x+y+z$.

It is clear that
$W_1(11)=W_2((0,0))=xyzpqrs=P$.
When $n=2$, the weights of elements in $\mq_2$ and $\operatorname{CQ}_2$ can be listed as follows:
\begin{equation*}\label{w3}
{\underbrace{2211\leftrightarrow (0,0)(1,1)}_{xy^2z^2pqr^2s=Pyzr},~\underbrace{1221\leftrightarrow (0,0)(1,2)}_{x^2yz^2pq^2rs=Pxzq},~\underbrace{1122\leftrightarrow (0,0)(1,3)}_{x^2y^2zp^2qrs=Pxyp}},
\end{equation*}
and the sum of weights is given by $P(xyp+xzq+yzr)=PP_1$.


Given $C_n=(0,0)(a_1,b_1)(a_2,b_2)\cdots (a_{n-1},b_{n-1})\in \operatorname{CQ}_n$.
Consider the elements in $\operatorname{CQ}_{n+1}$ generated from $C_n$ by appending the $2$-tuples $(a_n,b_n)$, where $1\leqslant a_n\leqslant n$ and $1\leqslant b_n\leqslant 3$.
Let $T$ be the corresponding simplified ternary increasing tree of $C_n$.
We can add $n+1$ to $T$ as a child of a vertex, which is not of degree three. Let $T'$ be the resulting simplified ternary increasing tree.
We first give a labeling of $T$ as follows. Label a leaf of $T$ by $P$, a degree
one vertex by $P_1$, a degree two vertex by $P_2$ and a degree three vertex by $t$. It is clear that
the contribution of any leaf to the weight is $xyzpqrs$, so we set $P=xyzpqrs$.

The $2$-tuples $(a_n,b_n)$ can be divided into three classes:
\begin{itemize}
  \item if $a_n\neq a_i$ for all $1\leqslant i\leqslant n-1$, then the node $n+1$ will be the only child of a leaf of $T$.
 This operation corresponds to the following change of weights:
  \begin{equation}\label{G1}
 W_2(C_n)\rightarrow W_2(C_{n+1})=W_2(C_n)(xyp+xzq+yzr),
  \end{equation}
which yields the substitution $P\rightarrow PP_1$.
The contribution of any leaf to the weight is $xyzpqrs$ and that of a degree one vertex is $xyp+xzq+yzr$ (which represents that this vertex may has a left child, a middle child or a right child).
When we compute the corresponding weights of Stirling permutations, it follows from~\eqref{G1} that we need to set
$P_1=xyp+xzq+yzr$.
  \item if there is exactly one 2-tuple $(a_i,b_i)$ in $C_n$ such that $a_n=a_i$, then the node $n+1$ will be the second child of the node $a_i$.
  There are six cases to add $n+1$.
This operation corresponds to the substitution $P_1\rightarrow 2PP_2$. Since we have six cases to insert $n+1$ and the sum of increased weights is $2(x+y+z)$, so we set
$P_2=x+y+z$.
  \item if there are exactly two 2-tuples $(a_i,b_i)$ and $(a_j,b_j)$ in $C_n$ such that $a_n=a_i=a_j$ and $i<j$, then the node $n+1$ will be the third child of $a_i$, and $n+1$ becomes a leaf with label $P$. This operation corresponds to the substitution $P_2\rightarrow 3tP$.
\end{itemize}
The aforementioned three cases exhaust all the possibilities to construct $\operatorname{SP}$-codes of length $n+1$
from a $\operatorname{SP}$-code of length $n$ by appending $2$-tuples $(a_n,b_n)$. In conclusion, each case corresponds to an application of the
substitution rules defined by the following grammar:
\begin{equation}\label{Iw3}
I=\{P\rightarrow PP_1,~P_1\rightarrow 2PP_2,~P_2\rightarrow 3tP\}.
\end{equation}

Note that the sum of degrees of all vertices in a simplified ternary increasing tree in
$\mtn_n$ is $n$.
Setting $w=P$, $v=P_1$ and $u=P_2$, it follows from~\eqref{Chen22} that
\begin{equation}\label{Chen2201}
Q_n=D_G^n(x)=D_I^{n-1}(P)=t^n\sum_{i+2j+3k=2n+1}\gamma_{n,i,j,k}\left(\frac{P_2}{t}\right)^i\left(\frac{P_1}{t}\right)^j\left(\frac{P}{t}\right)^k.
\end{equation}
Upon taking $P=xyzpqrs,~P_1=xyp+xzq+yzr$ and $P_2=x+y+z$, we arrive at Theorem~\ref{thm01}.~$\qed$.
\subsection{Proof of Theorem~\ref{thm02}}
\hspace*{\parindent}

Recall that $Q_n:=Q_n(x,y,z,p,q,r,s,t)$.
Let $$R_n(P,P_1,P_2,t)=t^n\sum_{i+2j+3k=2n+1}\gamma_{n,i,j,k}\left(\frac{P_2}{t}\right)^i\left(\frac{P_1}{t}\right)^j\left(\frac{P}{t}\right)^k.$$
Since $\deg (P)+\deg (P_1)+\deg (P_2)+\deg (t)=n$ in any term of $R_n(P,P_1,P_2,t)$,
we define the polynomial $\widetilde{R}_n(x,y,z)$ such that
$$R_n(P,P_1,P_2,t)=P_2^n\widetilde{R}_n\left(\frac{P}{P_2},\frac{P_1}{P_2},\frac{t}{P_2}\right).$$
Upon taking $P=xyzpqrs,~P_1=xyp+xzq+yzr$ and $P_2=x+y+z$, it follows from~\eqref{Chen2201} that
\begin{equation}\label{QnRn}
Q_n=R_n(P,P_1,P_2,t)=(x+y+z)^n\widetilde{R}_n\left(\frac{xyzpqrs}{x+y+z},\frac{xyp+xzq+yzr}{x+y+z},\frac{t}{x+y+z}\right).
\end{equation}
Note that
$$Q_n(1,1,1,p,q,r,s,t)=R_n(pqrs,p+q+r,3,t)=3^n\widetilde{R}_n\left(\frac{pqrs}{3},\frac{p+q+r}{3},\frac{t}{3}\right).$$
Substituting $$p\rightarrow\frac{3xyp}{x+y+z},~q\rightarrow\frac{3xzq}{x+y+z},~r\rightarrow\frac{3yzr}{x+y+z},~s\rightarrow\frac{s(x+y+z)^2}{9xyz},~t\rightarrow \frac{3t}{x+y+z},$$
we find that
$$\frac{1}{3}pqrs\rightarrow \frac{xyzpqrs}{x+y+z},~\frac{p+q+r}{3}\rightarrow \frac{xyp+xzq+yzr}{x+y+z},~\frac{1}{3}t\rightarrow \frac{t}{x+y+z}.$$
It follows from~\eqref{QnRn} that
$$Q_n=\left(\frac{x+y+z}{3}\right)^nQ_n\left(1,1,1,\frac{3xyp}{x+y+z},\frac{3xzq}{x+y+z},\frac{3yzr}{x+y+z},\frac{s(x+y+z)^2}{9xyz},\frac{3t}{x+y+z}\right),$$
as desired. Since $C_n(x,y,z)=Q_n(x,y,z,1,1,1,1,1)$, we get
$$C_n(x,y,z)=\left(\frac{x+y+z}{3}\right)^nQ_n\left(1,1,1,\frac{3xy}{x+y+z},\frac{3xz}{x+y+z},\frac{3yz}{x+y+z},\frac{(x+y+z)^2}{9xyz},\frac{3}{x+y+z}\right),$$
which yields~\eqref{thm02cor2}. The completes the proof.~$\qed$.
\subsection{Proof of Theorem~\ref{thm03}}\label{subsection34}
\hspace*{\parindent}

Given a $\operatorname{SP}$-code $C_n=((0,0),(a_1,b_1),(a_2,b_2)\ldots,(a_{n-1},b_{n-1}))$.
We make the symbols:
\begin{align*}
\fbox{j}&=\#\{a_i\mid (a_i,j)\notin C_n\},\\
\fbox{$j_1,j_2$}&=\#\{a_i\mid (a_i,j_1)\notin C_n~\&~(a_i,j_2)\notin C_n\},\\
\fbox{$j_1,j_2,j_3$}&=\#\{a_i\mid (a_i,j_1)\notin C_n~\&~(a_i,j_2)\notin C_n~\&~(a_i,j_3)\notin C_n\},\\
\doublebox{j}&=\#\{a_i\mid (a_i,j)\in C_n\},\\
\doublebox{$j_1,j_2$}&=\#\{a_i\mid (a_i,j_1)\in C_n~\&~(a_i,j_2)\in C_n\},\\
\doublebox{$j_1,j_2,j_3$}&=\#\{a_i\mid (a_i,j_1)\in C_n~\&~(a_i,j_2)\in C_n~\&~(a_i,j_3)\in C_n\},\\
\fbox{$j_1$}\doublebox{$j_2$}&=\#\{a_i\mid (a_i,j_1)\notin C_n~\&~ (a_i,j_2)\in C_n\},\\
\fbox{$j_1$}\doublebox{$j_2,j_3$}&=\#\{a_i\mid (a_i,j_1)\notin C_n~\&~ (a_i,j_2)\in C_n~\&~ (a_i,j_3)\in C_n\}.
\end{align*}
Recall that $$\alpha(\sigma)=\alpha_1^{\operatorname{apap}(\sigma)}\alpha_2^{\operatorname{dpa}(\sigma)}\alpha_3^{\operatorname{pdpd}(\sigma)},~
\beta(\sigma)=\beta_1^{\dplat(\sigma)}\beta_2^{\dasc(\sigma)}\beta_3^{\operatorname{dd}(\sigma)}\beta_4^{\operatorname{uu}(\sigma)}
\beta_5^{\ddes(\sigma)}\beta_6^{\pasc(\sigma)}.$$
The weight $W_3$ of $\sigma\in\mqn$ is defined by
$$W_3:=\alpha(\sigma)\beta(\sigma)x^{\asc{(\sigma)}}y^{\plat{(\sigma)}}z^{\des(\sigma)}p^{\lap(\sigma)}
q^{\eudd(\sigma)}r^{\rpd(\sigma)}s^{\operatorname{apd}(\sigma)}t^{\operatorname{vv}(\sigma)}.$$
It follows from Table~\ref{Table1} that the corresponding weight $W_4$ of $C_n\in\operatorname{CQ}_n$ is given as follows:
\begin{equation}\label{subsection34}
W_4:=\alpha(C_n)\beta(C_n)x^{{\fbox{1}}}
y^{\fbox{2}}z^{\fbox{3}}p^{\fbox{1,2}}q^{\fbox{1,3}}r^{\fbox{2,3}}s^{\fbox{1,2,3}}t^{\doublebox{1,2,3}}.
\end{equation}
where
$$\alpha(C_n)=\alpha_1^{\fbox{$1$}\doublebox{$2,3$}}\alpha_2^{\fbox{$2$}\doublebox{$1,3$}}\alpha_3^{\fbox{$3$}\doublebox{$1,2$}},$$
$$\beta(C_n)=\beta_1^{\fbox{$2$}\doublebox{$1$}}\beta_2^{\fbox{$1$}\doublebox{$2$}}\beta_3^{\fbox{$3$}\doublebox{$1$}}\beta_4^{\fbox{$1$}\doublebox{$3$}}
\beta_5^{\fbox{$3$}\doublebox{$2$}}\beta_6^{\fbox{$2$}\doublebox{$3$}}.$$

Let $T$ be the corresponding simplified ternary increasing tree of $C_n$.
A labeling of $T$ as follows. Label a leaf of $T$ by $\delta$, a degree
one vertex by $\delta_1$, a degree two vertex by $\delta_2$ and a degree three vertex by $t$.
Consider all the possibilities to construct $\operatorname{SP}$-codes of length $n+1$
from a $\operatorname{SP}$-code of length $n$ by appending $2$-tuples $(a_n,b_n)$. In the same way as in the proof of Theorem~\ref{thm01},
each case corresponds to an application of the
substitution rules defined by the following grammar:
\begin{equation*}\label{Iw3}
J=\{\delta\rightarrow \delta \delta_1,~\delta_1\rightarrow 2\delta \delta_2,~\delta_2\rightarrow 3t\delta\},
\end{equation*}
where $\delta=xyzpqrs,~\delta_1=\beta_4\beta_6xyp+\beta_2\beta_5xzq+\beta_1\beta_3yzr$ and $\delta_2=\alpha_1\beta_2\beta_4x+\alpha_2\beta_1\beta_6y+\alpha_3\beta_3\beta_5z$.
Setting $w=\delta$, $v=\delta_1$ and $u=\delta_2$, as in~\eqref{Chen2201}, we get
\begin{equation*}\label{Chen22012}
F_n=D_J^{n-1}(\delta)=t^n\sum_{i+2j+3k=2n+1}\gamma_{n,i,j,k}\left(\frac{\delta_2}{t}\right)^i\left(\frac{\delta_1}{t}\right)^j\left(\frac{\delta}{t}\right)^k.
\end{equation*}
Then upon taking $\delta=xyzpqrs,~\delta_1=\beta_4\beta_6xyp+\beta_2\beta_5xzq+\beta_1\beta_3yzr$ and $\delta_2=\alpha_1\beta_2\beta_4x+\alpha_2\beta_1\beta_6y+\alpha_3\beta_3\beta_5z$, we arrive at~\eqref{Fn-expansion}. This completes the proof.~$\qed$.

\section*{Acknowledgements}
The first author was supported by the National Natural Science Foundation of China (Grant number 12071063) and
Taishan Scholars Program of Shandong Province (No. tsqn202211146).
The third author was supported by the National science and technology council (Grant number: MOST 112-2115-M-017-004).
\bibliographystyle{amsplain}

\end{document}